\documentclass[12pt]{article}

\usepackage[utf8]{inputenc} 
\usepackage[margin=1.in]{geometry} 
\usepackage{graphicx} 
\usepackage[dvipsnames]{xcolor}
\usepackage{booktabs} 
\usepackage{array} 
\usepackage{verbatim} 
\usepackage[caption=false]{subfig}
\usepackage{amsthm,amsmath,amssymb}
\usepackage[english]{babel}
\usepackage{graphicx}
\usepackage[small,bf,hang]{caption} 
\usepackage{hyperref}

\newtheorem{theorem}{Theorem $\!\!$}[section]
\newtheorem{lemma}[theorem]{Lemma $\!\!$}
\newtheorem{remark}[theorem]{Remark $\!\!$}
\newtheorem{corollary}[theorem]{Corollary $\!\!$}

\begin{document}

\title{The Minimality of the Georges--Kelmans Graph}


\author{Gunnar Brinkmann\\
Department of Applied Mathematics, Computer Science \& Statistics,\\
  Ghent University, 9000 Ghent, Belgium\\
  Gunnar.Brinkmann@ugent.be\\
  \and
 Jan Goedgebeur\footnote{Supported by a Postdoctoral Fellowship of the Research Foundation Flanders (FWO).}\\Department of Applied Mathematics, Computer Science \& Statistics,\\
  Ghent University, 9000 Ghent, Belgium\\ and \\ Computer Science Department,\\
  University of Mons, 7000 Mons, Belgium\\ and \\ 
  Department of Computer Science,\\
  KU Leuven Campus Kulak, 8500 Kortrijk, Belgium \\
  Jan.Goedgebeur@kuleuven.be\\
  \and
  Brendan D.\ McKay\footnote{Supported by a Francqui International Professorship.}\\Department of Computing,\\
  Australian National University, Canberra ACT 2601, Australia\\ Brendan.McKay@anu.edu.au}
  

\maketitle

{\bf keywords: }  hamiltonian cycle, bipartite, cubic, connectivity, Tutte's conjecture, Barnette's conjecture, genus


\begin{abstract}

In 1971, Tutte wrote in an article that {\em it is tempting to conjecture that
 every 3-connected bipartite cubic graph is hamiltonian}. Motivated by this remark,
Horton constructed a counterexample on $96$ vertices.
In a sequence of articles by different authors several smaller counterexamples 
were presented. The smallest of these graphs is a graph on 50 vertices which was discovered
independently by Georges and Kelmans. In this article we show that there is no smaller 
counterexample. As all non-hamiltonian 3-connected bipartite cubic graphs in the
literature have cyclic 4-cuts---even if they have girth 6---it is natural to ask whether this
is a necessary prerequisite. In this article we answer this question in the negative and give
a construction of an infinite family of non-hamiltonian cyclically 5-connected  
bipartite cubic graphs.

In 1969 Barnette gave a weaker version of the conjecture stating that
 3-connected planar bipartite cubic graphs are hamiltonian. We show that Barnette's
conjecture is true up to at least 90 vertices.
We also report that a search of small non-hamiltonian 3\nobreakdash-connected bipartite cubic
graphs did not find any with genus less than~4.
\end{abstract}

\section{Introduction}

In this article all graphs will be simple and undirected unless
explicitly stated otherwise. For any standard graph theoretical
concepts not explicitly defined here, please refer
to~\cite{graph_theory_diestel}.

Tait conjectured in 1880 that every 3-connected planar cubic graph is
hamiltonian. This conjecture was disproved in 1946 by Tutte~\cite{tutte1946hamiltonian}, who
constructed a counterexample on 46 vertices. If Tait's conjecture had
been true, it would have implied the Four Colour Theorem. In the years
that followed, several researchers constructed smaller
counterexamples. In 1988, Holton and McKay~\cite{holton1988smallest}
completely settled the problem by showing that the smallest
non-hamiltonian 3-connected planar cubic graphs have 38 vertices and
that there are exactly 6 such graphs of that order (one of which is
the famous Barnette--Bosák--Lederberg graph).

For any graph $H$, $V(G)$ is its vertex set and $E(G)$ is its edge set.
A graph is \emph{cyclically $k$-edge-connected} if the
deletion of fewer than $k$ edges does not create two components both
of which contain at least one cycle.
Similarly, for  \emph{cyclically $k$-vertex-connected}.
With the exception of a few
graphs on 6 or fewer vertices, the cyclic edge-connectivity
and cyclic vertex-connectivity of a connected
cubic graph both equal the size of the smallest cut of 
independent edges~\cite{mccuaig}.
(Edges are \emph{independent} if they have no endpoints in common.)
Consequently we will call this number just \emph{cyclic connectivity}.
With the same small exceptions, the complement of a shortest cycle
has too many edges to be acyclic, so the cyclic connectivity is at most equal to the girth.
Also note that cyclic 3-connectivity implies 3-connectivity.

In 1971, Tutte~\cite{tutte19712} wrote that {\em it is tempting to conjecture that
 every 3-connected \textit{bipartite} cubic graph is hamiltonian}.
This statement that is often cited as a conjecture was disproved in 1976 by
Horton~\cite[p.\ 240]{bondy1976graph}, who constructed a counterexample
on 96 vertices. In 1982 he found a smaller counterexample, on 92
vertices~\cite{horton1982two}. Later Ellingham~\cite{ellinghamthesis}
discovered an infinite family of non-hamiltonian 3-connected
bipartite cubic graphs. The smallest member of his family has 78
vertices.

All of the examples mentioned so far have cyclic connectivity~3.
In~1983, Ellingham and Horton~\cite{ellingham1983non}
published a non-hamiltonian cyclically 4-connected bipartite cubic graph.
Their 54-vertex graph can be seen in Figure~\ref{fig:EHGK}(a).

Finally, Georges~\cite{georges1989non} and
Kelmans~\cite{kel1988cubic,kelmans1994constructions} independently
discovered another infinite family of non-hamiltonian cyclically 4-connected
bipartite cubic graphs, the smallest of which has 50 vertices.  
Their 1986 submission dates were only 15 days apart.
This graph is shown in Figure~\ref{fig:EHGK}(b) and is the smallest known.

In~\cite{kel1988cubic} it is written that Lomonosov and Kelmans proved
without computer assistance that Tutte's conjecture holds up to 30 vertices,
but no reference is given. Recently, Knauer and Valicov~\cite{knauer2019cuts} verified Tutte's conjecture up to 40 vertices using computational methods.
Our aim in this paper is to prove that the Georges--Kelmans graph is in fact minimal.

\begin{theorem}\label{thm:main}
There is no smaller non-hamiltonian 3-connected bipartite cubic graph than the 
Georges--Kelmans graph (which has 50 vertices).
Moreover if there is another example with 50 vertices, it has girth and
cyclic connectivity exactly~6.
\end{theorem}

%

\begin{figure}[!htb]
\unitlength=\textwidth
\begin{picture}(1,0.36)(0,0)
 \put(0.015,0.05){\includegraphics[width=0.46\textwidth]{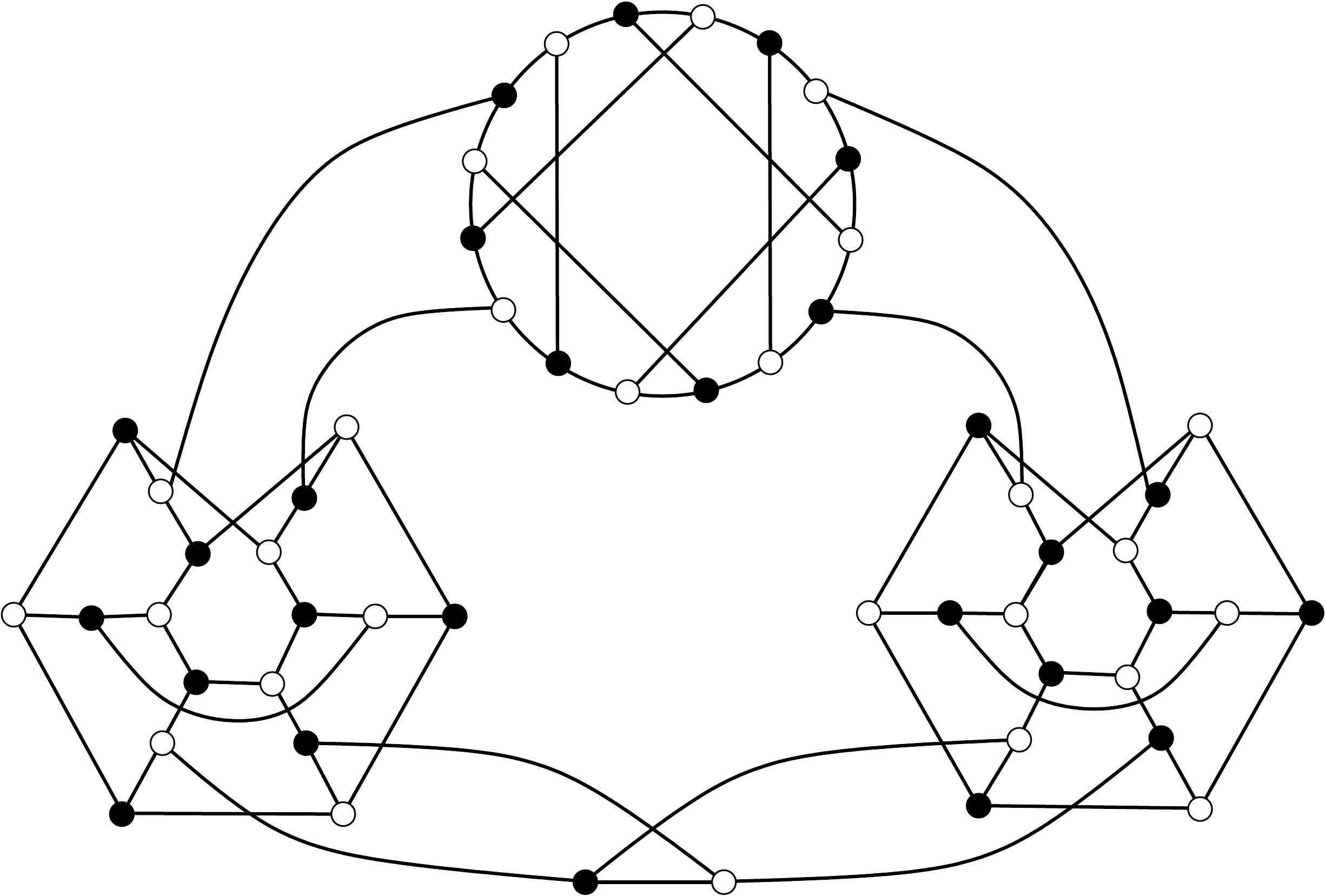}}
 \put(0.505,0.11){\includegraphics[width=0.49\textwidth]{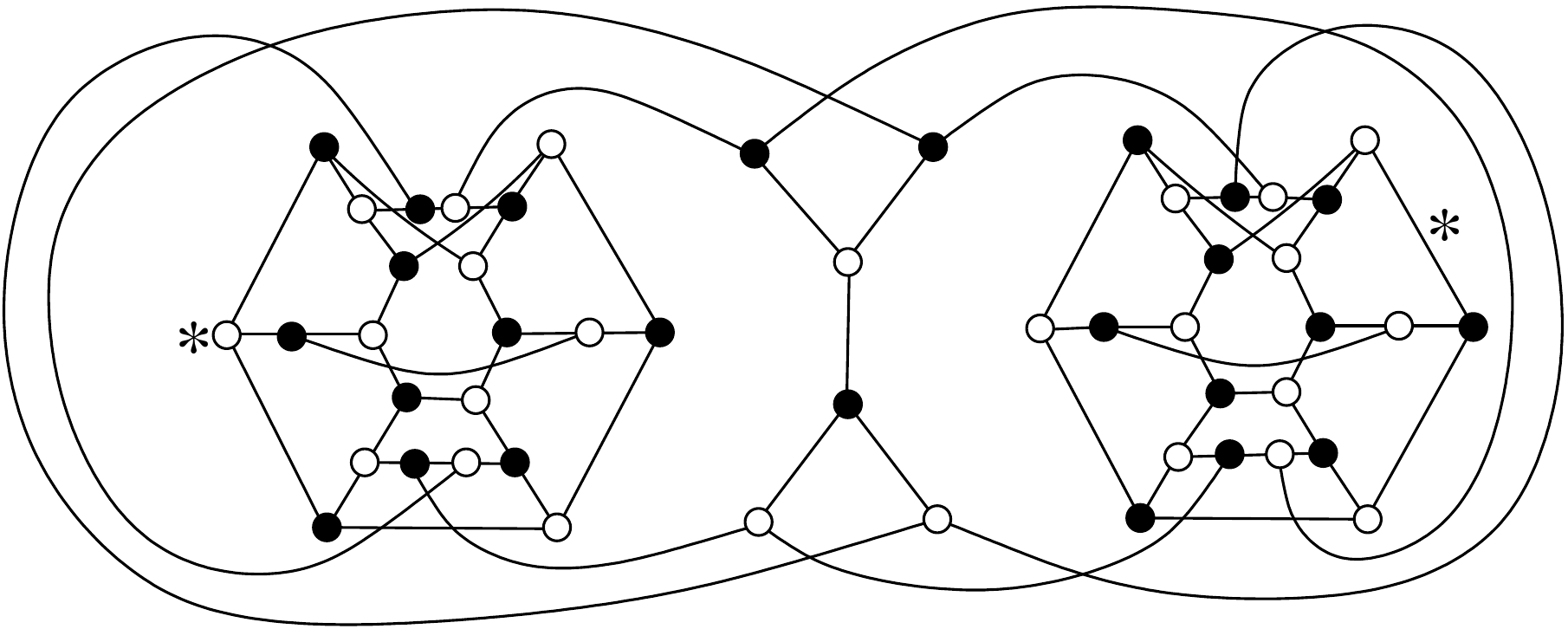}}
 \put(0.225,0.005){(a)}\put(0.75,0.006){(b)}
\end{picture}
\caption{The Ellingham--Horton graph on 54 vertices (a) and the Georges--Kelmans graph on 50 vertices (b).
(The asterisks are used in Section~\ref{sect:cyc5}.)
\label{fig:EHGK}}
\end{figure}

Since we plan on liberal use of the computer, the reader may wonder why we
do not just test all the 3-connected bipartite cubic graphs up to 48 vertices.
The reason is that there are far too many of them.  Based on exact counts for smaller
sizes (Table~\ref{table:cubic_bip_graphs} below; see~\cite{HoG} for some
of the graphs) we estimate that there
are about $5\times 10^{17}$ cubic bipartite graphs with 48 vertices, most of
them 3-connected. 
With our programs able to generate and test about
50,000 bipartite cubic graphs per second, this approach
would take about 320,000 CPU years.

\begin{table}[ht!b]
\centering
\begin{tabular}{c | c  c  c}
Vertices & Girth at least 4 & Girth at least 6 & Girth at least 8\\
		\midrule
6 & 1 &  &  \\
8 & 1 &  &  \\
10 & 2 &  &  \\
12 & 5 &  &  \\
14 & 13 & 1 &  \\
16 & 38 & 1 &  \\
18 & 149 & 3 &  \\
20 & 703 & 10 &  \\
22 & 4\,132 & 28 &  \\
24 & 29\,579 & 162 &  \\
26 & 245\,627 & 1\,201 &  \\
28 & 2\,291\,589 & 11\,415 &  \\
30 & 23\,466\,857 & 125\,571 & 1 \\
32 & 259\,974\,248 & 1\,514\,489 & 0 \\
34 & 3\,087\,698\,618 & 19\,503\,476 & 1 \\
36 & 39\,075\,020\,582 & 265\,448\,847 & 3 \\
38 & 524\,492\,748\,500 & 3\,799\,509\,760 & 10 \\
40 & ~7\,439\,833\,931\,266~ & 57\,039\,155\,060 & 101 \\
42 & ? & 896\,293\,917\,129 & 2\,510 \\
44 & ? & ? & 79\,605 \\
46 & ? & ? & 2\,607\,595 \\
48 & ? & ? & 81\,716\,416 \\
50 & ? & ? & 2\,472\,710\,752 \\
\end{tabular}
\caption{Counts of connected bipartite cubic graphs with girth at least 4, 6 or~8.}

\label{table:cubic_bip_graphs}
\end{table}

We will instead take an incremental approach that combines theory
and computation, successively eliminating the possibilities that a
non-hamiltonian 3-connected bipartite cubic graph smaller than
the Georges--Kelmans graph is cyclically 3-connected, cyclically
4-connected, cyclically 5-connected, and finally show that it doesn't exist.
The required amount of computer time was large but available.

%

Furthermore, noting that all previous examples had cyclic connectivity
at most 4, we construct an infinite family of non-hamiltonian
cyclically 5-connected bipartite cubic graphs, of which the
smallest has 300 vertices.

Another important strongly related conjecture is Barnette's
conjecture, which is a weakened combination of Tait and Tutte's
conjectures. Barnette~\cite{barnette1969conjecture} conjectured in
1969 that every 3-connected \textit{planar} bipartite cubic graph is
hamiltonian. Even after 50 years, this conjecture is still open. 
Holton, Manvel and McKay~\cite{holton1985hamiltonian} showed
that Barnette's conjecture is true up to at least 64 vertices and this
was later improved to 84 vertices by Aldred, Brinkmann and
McKay~\cite{announcementbarnette}. We will show that Barnette's
conjecture is true up to at least 90 vertices.
We also did a non-exhaustive search for non-hamiltonian 3-connected
bipartite cubic graphs that are close to planar in the sense of having
low genus.  However we did not find any with genus less than~4.

The article is organised as follows. In Section~\ref{sect:properties}
we give the theoretical and computational results necessary to prove
Theorem~\ref{thm:main}.  
We verified the correctness of each of our computational results by making an independent implementation of every program which was specifically developed for this project (often using an alternative algorithm) and using it to replicate each computational result
as far as CPU time limits allowed.
These correctness tests are described in Section~\ref{sect:properties} as well.
In Section~\ref{sect:cyc5} we present an
infinite family of non-hamiltonian cyclically 5-connected
bipartite cubic graphs.  In Section~\ref{sect:barnette_genus} we
verify Barnette's conjecture up to 90 vertices and
give the results about the generation of the non-hamiltonian 3-connected bipartite cubic
graphs constructed in this project.  
As this project required the discovery of more than $10^{14}$ hamiltonian
cycles in cubic graphs of up to 50 vertices, we needed a very efficient practical algorithm for this. We describe this algorithm in Section~\ref{sect:cubhamg}.


\section{Properties of a minimal non-hamiltonian bipartite cubic 3-connected graph}
\label{sect:properties}

\begin{lemma}\label{lem:parity}
  In a bipartite graph with vertices of degree 2 and 3,
  the number of vertices of degree~2 of each colour (that is: of each bipartition class)
  are equal modulo~$3$.
\end{lemma}
\begin{proof}
 Define $n_{d,c}$ to be the
 number of vertices of degree $d$ and colour~$c$, for
 $d=2,3$ and $c=0,1$.  
 Counting the edges in two different ways, we have
 $3n_{3,0}+2n_{2,0}=3n_{3,1}+2n_{2,1}$, from which it follows that $n_{2,0}$
 and $n_{2,1}$ are equal modulo~3.
\end{proof}

\begin{lemma}\label{lem:cyc3}
 A minimal non-hamiltonian 3-connected bipartite cubic graph is
 cyclically 4-connected.
\end{lemma}

\begin{proof}
 Suppose that a non-hamiltonian 3-connected bipartite cubic graph $G$
 has an independent 3-edge cut $\{e_1,e_2,e_3\}$.
 Divide $G$ into two parts at the cut. Due to Lemma~\ref{lem:parity},
 in each part the vertices of degree $2$ have the same colour, so
 we can use one extra vertex for each part and connect it to
the three vertices of degree 2 to form two 
bipartite cubic graphs $G_1,G_2$. It is an easy consequence of Menger's theorem that they are both 3-connected. 
For $j=1,2$ we label the new edges $\{e_{1,j},e_{2,j},e_{3,j}\}$ 
so that for $1\le i\le 3$, $e_{i,j}$ is in $G_j$ and has one
endpoint the same as $e_i$.
As the cut was
independent, $G_1$ and $G_2$ are smaller than $G$.
 
 If one of $G_1,G_2$ was non-hamiltonian,
 $G$ would not be minimal, so we may assume that both
$G_1$ and $G_2$ are hamiltonian.

Now assume that one smaller graph, w.l.o.g. $G_1$, has a hamiltonian cycle 
not containing~$e_{i,1}$. If $G_2$ had a hamiltonian 
cycle not containing $e_{i,2}$, these cycles could be combined to form a
hamiltonian cycle of $G$.
If for one smaller graph $G_j$, each edge of $\{e_{1,j},e_{2,j},e_{3,j}\}$ 
could be avoided by a hamiltonian cycle or for both smaller graphs there would be at least
two edges that can be avoided, there would be a combination of hamiltonian cycles where
the same edge would be avoided---so they could be combined to form a hamiltonian cycle
of $G$. So in one smaller graph at most one edge can be avoided (which means that this
{\em forbidden} edge is in no hamiltonian cycle) and in the other at most two edges can be avoided
(which means that there is a {\em forced} edge that lies in each hamiltonian cycle).
 
 A straightforward
 computation using the programs {\em minibaum} and {\em cubhamg}
showed that forced edges first appear at 30 vertices and forbidden
 edges first appear at 34 vertices.  (\emph{Minibaum}~\cite{brinkmann96} is a generator for cubic graphs which can also generate bipartite cubic graphs efficiently and {\em cubhamg} is described in Section~\ref{sect:cubhamg}.
 The bipartite cubic graphs up to 30 vertices are available at the House of Graphs~\cite{HoG}.)
 Using an independent implementation of a program to test if a graph contains forced or forbidden edges, we obtained exactly the same number of graphs with forced/forbidden edges up to 34 vertices. Therefore, this construction only
 yields non-hamiltonian graphs with at least 62 vertices---larger than
 the Georges--Kelmans graph.
\end{proof}

\begin{lemma}\label{lem:cyc4}
 A minimal non-hamiltonian 3-connected bipartite cubic graph
 cannot have cyclic connectivity~$4$ unless it is the
 Georges--Kelmans graph.
Moreover, a minimal non-hamiltonian 3-connected bipartite cubic graph
has girth at least 6.
\end{lemma}

\begin{proof}
For a cubic graph $H$, define two types of edge $e$:
\vspace{-0.5em}
\begin{itemize}
\addtolength{\itemsep}{-2mm}
\item Type 1: There is a hamiltonian cycle through at most 3 of the 4
paths of three edges whose central edge is~$e$.
\item Type 2: At least one of the auxiliary (non-bipartite) graphs $H',H''$ defined as
in Figure~\ref{fig:split4-type2} has no hamiltonian cycle.
\end{itemize}

\begin{figure}[h!t]
	\centering
	\includegraphics[width=0.45\textwidth]{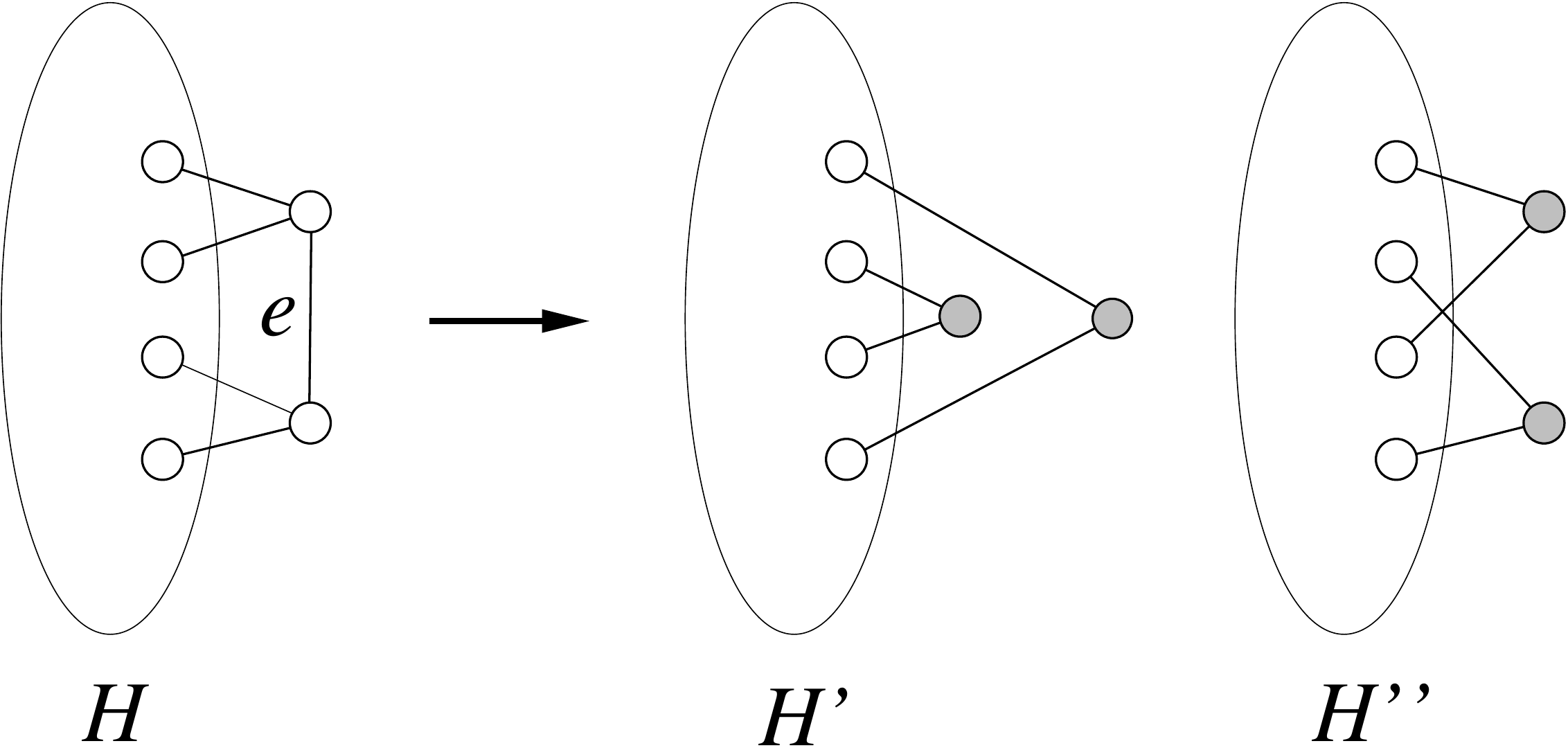}
	\caption{Auxiliary graphs for type 2 edges used in the proof of Lemma~\ref{lem:cyc4}.}
	\label{fig:split4-type2}
\end{figure}

\begin{figure}[h!t]
	\centering
	\includegraphics[width=0.6\textwidth]{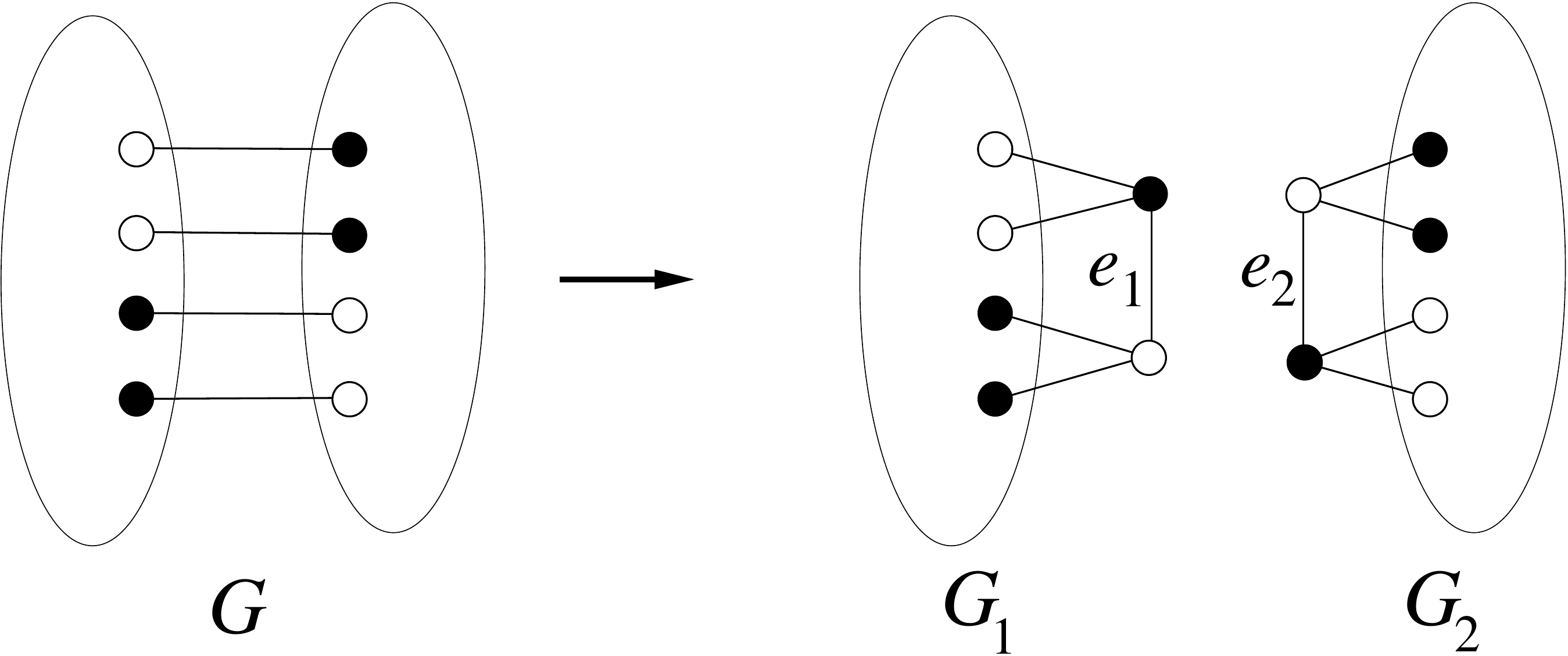}
	\caption{Splitting a graph into two at a 4-edge cut.}
	\label{fig:4cut}
\end{figure}

Now let $G$ be a minimal non-hamiltonian bipartite cubic graph  with
 cyclic connectivity~4.  
 By Lemma~\ref{lem:parity}, we can divide $G$ into two parts at a 4-edge
 cut and use new edges $e_1,e_2$ to complete
 the parts into bipartite cubic 3-connected graphs
 $G_1,G_2$ as shown in Figure~\ref{fig:4cut}.  The 3-connectivity of
 $G_1$ and $G_2$ follows from the observation that a 2-cut in either
 of them would imply a cyclic 3-edge cut in~$G$.

If $G_1$ has a hamiltonian cycle using $e_1$, then $e_2$ must have
type~1, since otherwise $G$ would be hamiltonian.
Similarly, if $G_1$ has a hamiltonian cycle avoiding $e_1$, then $e_2$
must have type~2, since otherwise $G$ would be hamiltonian.

A direct computation, again using {\em minibaum} and  {\em cubhamg},
showed that type~1 or type~2 edges first appear at 18 vertices, so we have
$|V(G_1)|\ge 18$, and similarly $|V(G_2)|\ge 18$.  Consequently,
it suffices to test combinations of graphs $G_1,G_2$ with
$18\le |V(G_1)|\le 26$ and $18\le |V(G_2)|\le 54-|V(G_1)|$
with $G_1$ restricted to graphs having an edge of type~1 or type~2.
We could also restrict $G_2$ in the same way, but the number of
possibilities for $G_1$ when $G_2$ is large is so small that simply
testing every graph as $G_2$ is as fast as checking $G_2$ for
edges of type~1 or type~2.

This computation yielded only the Georges--Kelmans graph.  The counts
are shown in Table~\ref{tab:cyc4} (which includes graphs that are connected
but not 3-connected). Using an independent implementation, 
we obtained exactly the same number of graphs with type 1 or type 2 edges
as in Table~\ref{tab:cyc4}. Using another independent program we again tested
all combinations of $G_1,G_2$ and this indeed only yielded the Georges--Kelmans graph.

For the second part of the lemma, it is only necessary to observe that the
Georges--Kelmans graph has girth~6.
\end{proof}

\begin{table}[h!t]
\centering
\begin{tabular}{cccc|cc}
  $|V(G_1)|$ & type 1 count & type 2 count & $G_1$ total & $|V(G_2)|$ & $G_2$ count \\
  \hline
  18 & 1 & 1 & 0 & 18--36 & 6\,461\,410\,120 \\
  20 & 1 & 1 & 2 & 18--34 & 3\,373\,711\,502\\
  22 & 5 & 3 & 8 & 18--32 & 286\,012\,884\\
  24 & 15 & 14 & 27 & 18--30 & 26\,038\,636\\
  26 & 71 & 56 & 121 & 18--28 & 2\,571\,779
\end{tabular}
\caption{The numbers of graphs that can act as $G_1,G_2$ in the proof of Lemma~\ref{lem:cyc4} with the candidates for $G_1$ also listed according to their type.}
\label{tab:cyc4}
\end{table}

At this stage we could consider finishing the proof of Theorem~\ref{thm:main}
by computation alone.  However, we estimate the number of bipartite cubic graphs
with girth 6 up to 48 vertices to be around $4.5\times 10^{15}$.  With a lower
generation rate of about 40,000 graphs per second, this still amounts to
3,500 CPU years. By Lemma~\ref{lem:cyc4}, we could also restrict our
search to cyclic connectivity at least~5, but the counts are not much less
and we don't know of a fast generator.

\begin{lemma}\label{lem:cyc5}
 A minimal non-hamiltonian 3-connected  bipartite cubic graph
 cannot have cyclic connectivity~$5$.
\end{lemma}

\begin{proof}
Define a \textit{5-piece} to be a connected bipartite graph of
girth at least 6, cubic apart from 5 vertices of degree~2.
By Lemma~\ref{lem:parity}, a 5-piece has a vertex of
degree~2 whose colour is different from the other vertices
of degree~2; call that the \textit{special vertex} of the 5-piece.
Also, the number of vertices in the 5-piece with the same
colour as the special vertex is one less than the number
with the other colour.

\begin{figure}[h!t]
	\centering
	\includegraphics[width=0.58\textwidth]{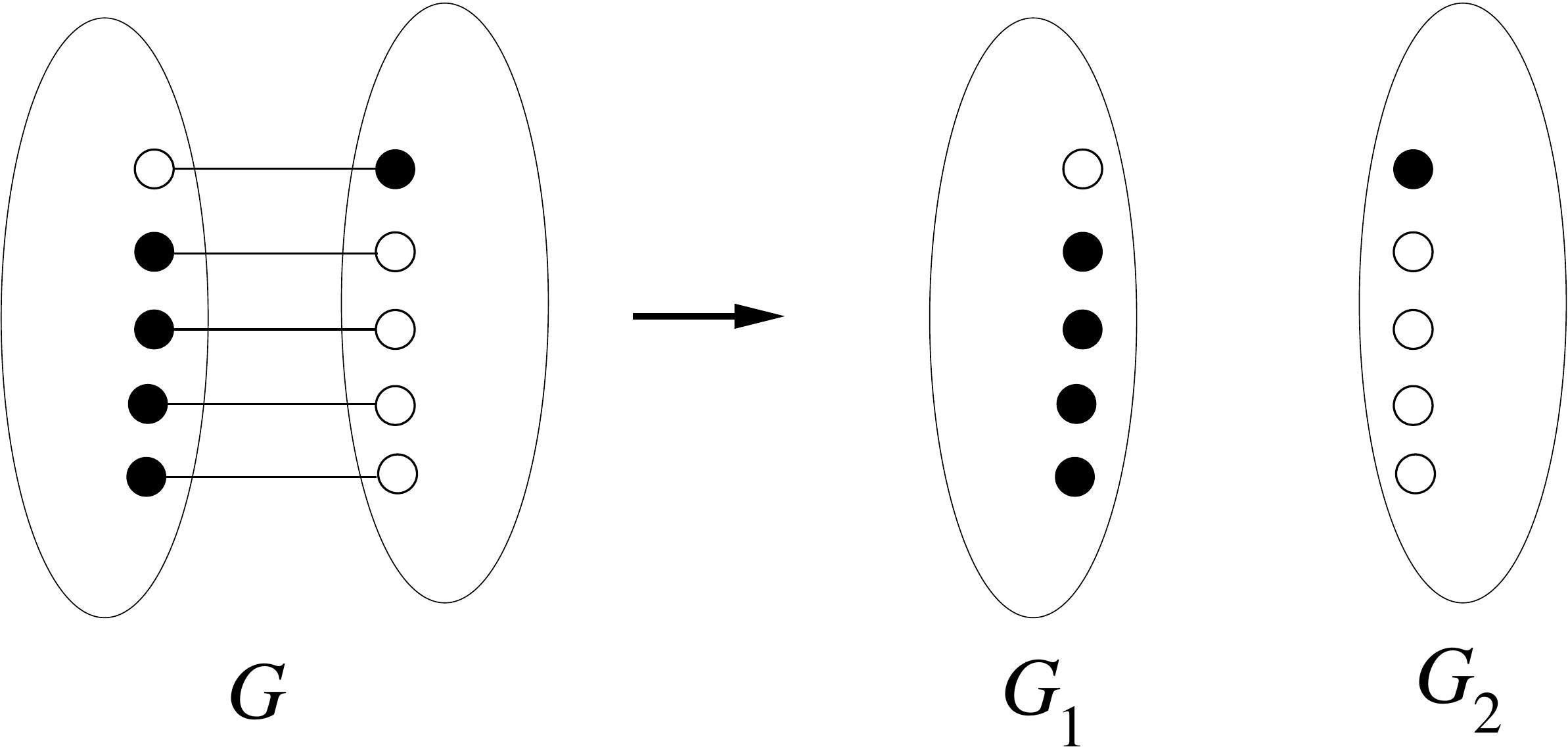}
	\caption{Splitting a graph into two at a 5-edge cut.}
	\label{fig:5cuta}
\end{figure}

\begin{figure}[h!t]
	\centering
	\includegraphics[width=0.62\textwidth]{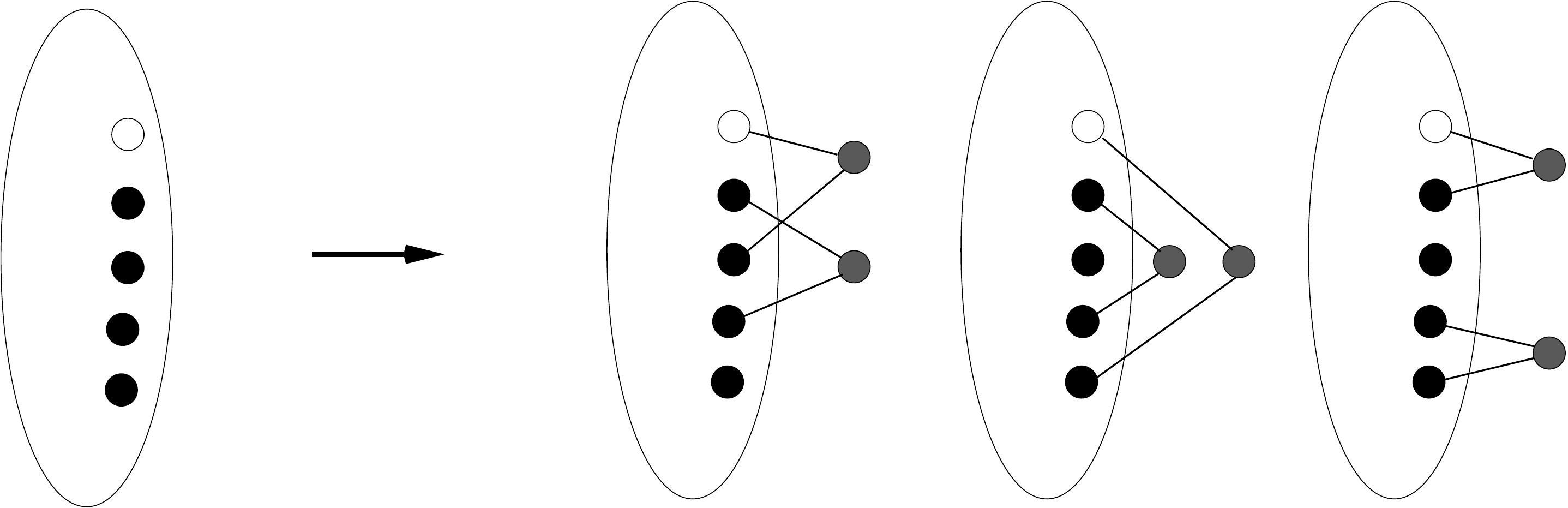}
	\caption{Three of the 12 test graphs of a 5-piece.}
	\label{fig:5cuta2}
\end{figure}

Let $G$ be a bipartite cubic  graph
 with cyclic connectivity~$5$ and at most 50 vertices -- so $G$ is a candidate for a counterexample to the lemma.
Separate $G$ at a 5-cut into 5-pieces $G_1$ and $G_2$ as
in Figure~\ref{fig:5cuta}.  We know that $G_1$ and $G_2$
are connected since otherwise $G$ would have an edge
cut of two independent edges.  The special vertices of
$G_1$ and $G_2$ are adjacent in~$G$; call that the
\textit{special edge} of~$G$.

The difference in the numbers of vertices in the two colour classes in 5-pieces imply:
\vspace{-0.7ex}
\begin{itemize}
\itemsep=0pt
\item[(a)] If a hamiltonian cycle in $G$ uses 4 edges of the cut, then
one of those edges is the special edge.
\item[(b)] If a hamiltonian cycle in $G$ uses only 2 edges of the cut, then
neither of those edges is the special edge.
\end{itemize}

Given a 5-piece $H$, a ``test graph'' for $H$ is formed by adjoining
two vertices of degree~2, together adjacent to four distinct vertices
of degree~2 in $H$, one of which is the special vertex.
There are 12 (non-bipartite) test graphs, three of which are shown in Figure~\ref{fig:5cuta2}.

Classify 5-pieces as follows:
\vspace{-0.5em}
\begin{itemize}
\addtolength{\itemsep}{-2mm}
\item Class 0: None of the test graphs is hamiltonian;
\item Class 1: At least one of the test graphs is hamiltonian;
\item Class 2: All of the test graphs are hamiltonian (a subset of Class 1).
\end{itemize}

\noindent\textit{Claim 1: } 
If $G_1$ and $G_2$ are in Class 1, with one of them in
Class 2, then $G$ is hamiltonian.

\noindent\textit{Proof: } Suppose that $G_2$ is in Class 2. 
Since $G_1$ is in Class 1, it can be covered
by two paths with endpoints $v_1,v_2$ and $w_1,w_2$,
where one of $v_1,v_2,w_1,w_2$ is the special vertex. 
Join these two paths into a hamiltonian cycle in $G$ using the
hamiltonian cycle in the test graph of $G_2$ where one vertex was connected
to $v_1,v_2$ and the other to $w_1,w_2$.

\medskip 

\noindent\textit{Claim 2: } If $|V(G_1)|\le 13$, then $G$ is hamiltonian.

\noindent\textit{Proof: }
As shown in Table~\ref{tab:cyc5} and checked twice by computer,
there are no 5-pieces of order less than 11,
one of order 11 and two of order~13.
They are depicted in Figure~\ref{fig:5cutb}.
All of them are in Class~2.
So, by Claim 1, the only possibility that $G$ is non-hamiltonian is for $G_2$
to be in Class~0.  A hamiltonian cycle in $G$ that uses 4 edges of the
cut would imply hamiltonicity of one of $G_2$'s test graphs, so the
only possibility is a hamiltonian cycle in $G$ that uses 2 edges of the cut
(none of them the special edge) and induces hamiltonian paths in
$G_1$ and $G_2$.

\begin{figure}[h!t]
	\centering
	\includegraphics[width=0.9\textwidth]{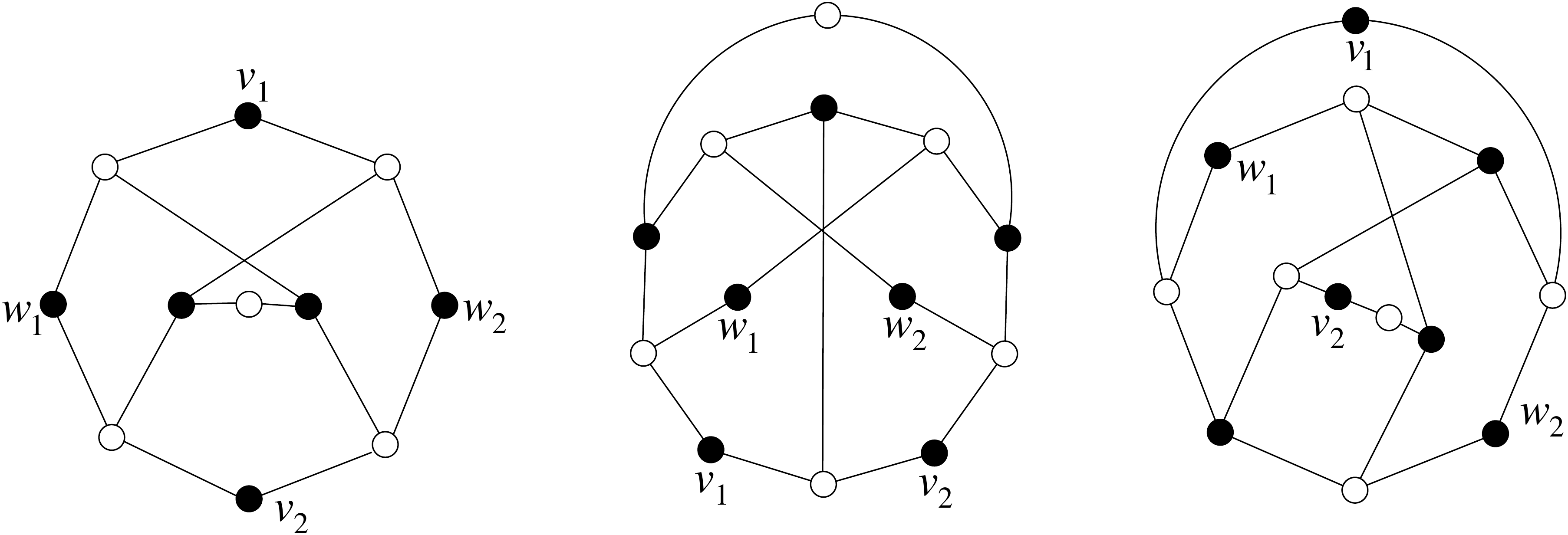}
	\caption{The smallest 5-pieces.}
	\label{fig:5cutb}
\end{figure}

Looking at the hamiltonian paths in $G_1$, we
find that the non-special vertices of degree~2 can be labelled
$v_1,v_2,w_1,w_2$ such that any pair of them can be joined
by a hamiltonian path except $v_1,v_2$ and possibly $w_1,w_2$.
(See Figure~\ref{fig:5cutb}.)
Let $v'_1,v'_2,w'_1,w'_2$ be the vertices of $G_2$ adjacent in $G$
to $v_1,v_2,w_1,w_2$, respectively.
Now construct a bipartite cubic graph $G^+_2$ from $G_2$ by adjoining a
path $xyz$ with $x$ adjacent to $v'_1$ and $v'_2$,
$y$ adjacent to the special vertex, and $z$ adjacent to $w'_1$ and $w'_2$.
Since $|V(G^+_2)| < |V(G)|$, and it is 3-connected since otherwise
$G$ would have a 3-edge cut, $G^+_2$ is hamiltonian by the
minimality of $G$.  Also, since $G_2$ is in Class~0, any hamiltonian
cycle in $G^+_2$ cannot use exactly one of the edges $xy$ and $yz$
(else one of the test graphs of $G_2$ is hamiltonian), so it  must use both
$xy$ and $yz$.  This provides a
hamiltonian path in $G_2$ from a vertex in $\{v'_1,v'_2\}$ to a vertex in $\{w'_1,w'_2\}$.
Any such path can be
combined with a hamiltonian path in $G_1$ to make a hamiltonian
cycle in~$G$.  This completes the proof of Claim~2.

\begin{table}[ht!]
\centering
\begin{tabular}{c|cc}
  $n$ & 5-pieces & $|\text{Class\,1} \setminus \text{Class 2}|$\\
  \hline
  11 & 1 & 0\\
  13 & 2 & 0\\
  15 & 12 & 2\\
  17 & 90 & 7\\
  19 & 754 & 14\\
  21 & 7\,003 & 25\\
  23 & 70\,639 & 68\\
  25 & 766\,134 & 251\\
  27 & 8\,862\,333 & 1\,086\\
  29 & 108\,917\,294 & 6\,098\\
  31 & 1\,417\,268\,482 & 44\,842\\
  33 & 19\,471\,253\,036 & 393\,423\\
  35 & ~281\,715\,327\,672 & 3\,887\,896
\end{tabular}
\caption{Computation for Lemma~\ref{lem:cyc5}.}
\label{tab:cyc5}
\end{table}

\medskip

Now we can complete the proof of the lemma.
We used the program {\em multigraph} to construct all 5-pieces and tested them with {\em cubhamg}. (\emph{Multigraph} can generate all simple graphs or multigraphs with a given degree sequence. It implements the same ideas as {\em minibaum}~\cite{brinkmann96} but for general
degree sequences, and as it does not contain new ideas, it was never published. For a large
number of degree sequences its results were tested and confirmed by 
the program described in~\cite{Gr93}.)
The computations showed that there are no 5-pieces of order at most 35
in Class~0.  The number in Class 1 $\setminus$ Class 2 are listed
in Table~\ref{tab:cyc5}.  The total computation time was about 5 CPU years.
We also independently determined all 5-pieces up to 29 vertices by using {\em minibaum} to generate all bipartite cubic graphs and using a separate program to remove one vertex and two of its neighbours in all possible ways of each input graph and retaining the connected graphs of girth at least 6 among the graphs resulting from this operation. We then determined the class of each 5-piece using another independent implementation and the results were in complete agreement with the counts reported in Table~\ref{tab:cyc5}.

We joined all combinations of two of these 5-pieces in Class 1 $\setminus$ Class 2 up to 50
vertices and found only
the Georges--Kelmans graph among those that were 3-connected. (Also here we replicated this result using an independent implementation of the joining program.)
The Georges--Kelmans graph has cyclic connectivity~4.
Together with Claims 1 and 2, this completes the proof.
\end{proof}


\begin{lemma} \label{lem:girth8}
All bipartite cubic graphs with girth at least 8 up to 50 vertices are hamiltonian.
\end{lemma}

\begin{proof}
 This result was purely computational.
 There are only about $2.5\times 10^{9}$ such graphs (see Table~\ref{table:cubic_bip_graphs})
 but the much reduced generation time for girth~8 meant that it took about 22 CPU years.
 
 After completing this long computation, we found that a modification of Meringer's program
 {\em genreg}~\cite{Meringer} made by the first author could generate the graphs in only 8 CPU months.
 This prompted us to perform a stronger test: all bipartite cubic graphs with girth at least 8 up to
 50 vertices have the property that for each pair of distinct edges $e,e'$ there is a hamiltonian
 cycle using~$e$ and not using~$e'$.
\end{proof}

\begin{proof}[Proof of Theorem~\ref{thm:main}]
Let $G$ be a non-hamiltonian 3-connected bipartite cubic graph with at
most 48 vertices. By Lemmas~\ref{lem:cyc3}--\ref{lem:girth8},
we know that $G$ is cyclically 6-connected and has girth~6.

Our approach is as follows: We define a reduction that transforms a bipartite cubic graph
with girth $6$ on $n$ vertices to a bipartite cubic graph with some 4-cycles on 
$n-8$ vertices. The set of reduced graphs will be much smaller than the set
of original graphs, but there will also be irreducible graphs on $n$ vertices.
The irreducible graphs have to be generated and tested directly and
on the reduced graphs some tests have to be performed in order to
guarantee that they do not come from a non-hamiltonian graph,
and reduced graphs that do not pass the test have to be extended and
checked for hamiltonicity.

In order to be able to reduce the computation time to an amount
available on a modern cluster, there has to be a balance between the
two parts.  On one hand the reduction should be so that there are not
too many irreducible graphs and that it is possible to generate all
irreducible graphs. On the other hand the tests necessary on the reduced graphs
must not be too expensive, so that the reduced graphs can be tested in an
affordable amount of time. The (admittedly very technical)
reduction we used is the following:

Let $\hat G$ be the bipartite graph with $16$ vertices depicted on the left hand side of
Figure~\ref{fig:8gon}. 
We will call a bipartite cubic graph $G$ of girth 6 
{\em reducible} if it contains $\hat G$ in a way that
has the following properties:

\begin{itemize}
\itemsep=0pt
\item[(i)] Neither $v'_8$ and $v'_3$ nor $v'_7$ and $v'_4$ are adjacent.

\item[(ii)] At least one of the paths $v'_1,v_1,v_2,v'_2$ and $v'_5,v_5,v_6,v'_6$ 
lies on a 6-cycle.

\item[(iii)] If exactly one of the paths in (ii) lies on a 6-cycle, then in addition
at least one of the paths $v'_8,v_8,v_1,v_2,v_3,v'_3$ and $v'_4,v_4,v_5,v_6,v_7,v'_7$
lies on an $8$-cycle.

\end{itemize}

The reduced graph $G_r$ is then obtained by deleting $v_1,\dots ,v_8$ and adding the edges
$e_1=\{v'_1,v'_2\}, e_2=\{v'_8,v'_3\}, e_3=\{v'_7,v'_4\}$, and $e_4=\{v'_6, v'_5\}$,
as depicted on the right hand side of Figure~\ref{fig:8gon}. 
It is obvious that $G_r$ is bipartite and cubic. Moreover, since
$G$ has girth 6 and satisfies property (i), $G_r$ is simple.
Since $G_r$ is cubic, each component of $G_r$ contains a cycle.
If cycles $Z_1$ and $Z_2$ lie in different components of $G_r$, then with Menger's theorem
applied to the graph obtained by adding two new vertices and connecting them once to
all vertices of $Z_1$ and once to all vertices of $Z_2$, the
cyclic 5-connectivity of $G$ means that there are at least 5 edge-disjoint
paths in $G$ from $Z_1$ to $Z_2$.  However, at most 4 of these paths
can contain vertices of the set $\{v_1,\dots ,v_8\}$, so one of them connects
$Z_1$ and $Z_2$ in $G_r$, contradicting the assumption that $Z_1$
and $Z_2$ are in different components.  Thus, $G_r$ is connected.

\begin{figure}[h!t]
	\centering
	\includegraphics[width=0.75\textwidth]{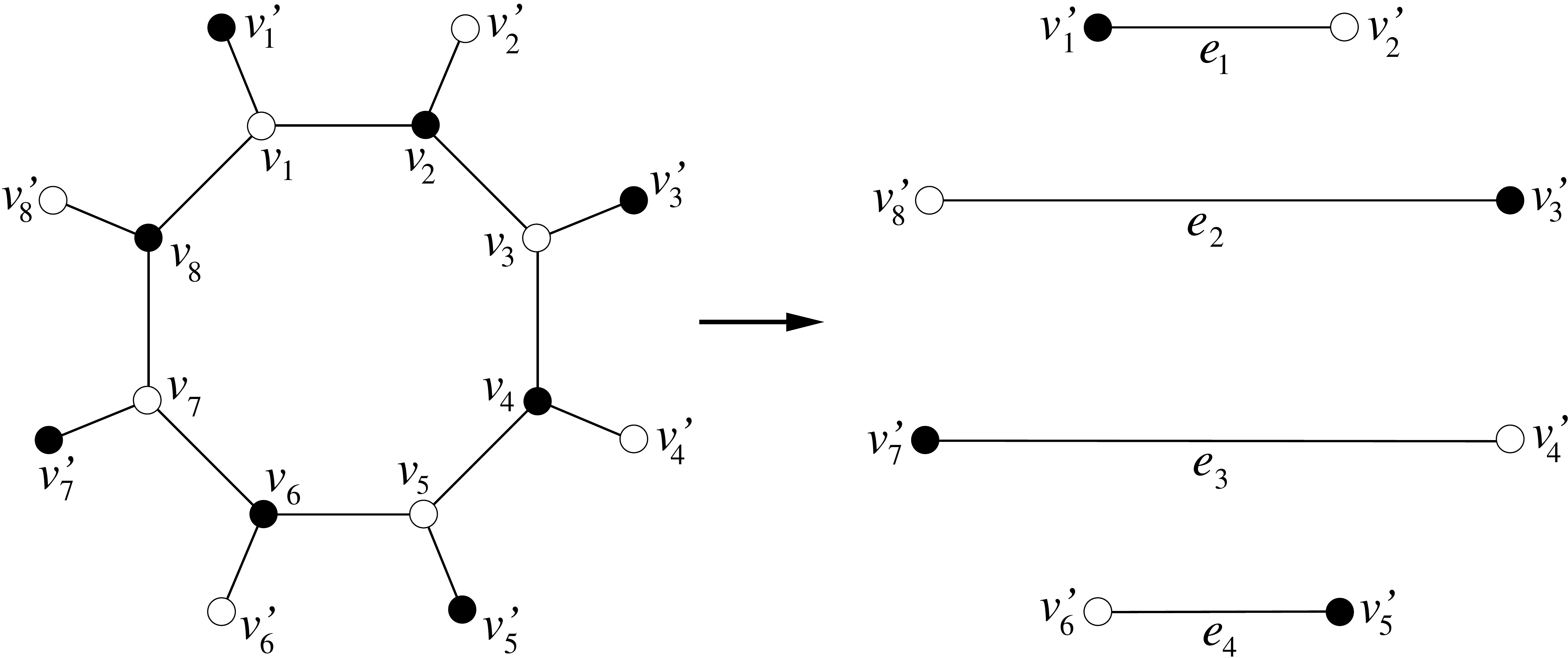}
	\caption{The reduction of an 8-cycle.}
	\label{fig:8gon}
\end{figure}

If $G_r$ is hamiltonian, this does not necessarily imply that $G$ is also hamiltonian,
but some hamiltonian cycles in $G_r$ imply hamiltonicity in $G$. We call a
hamiltonian cycle $H_r$ of $G_r$ {\em extendable} for $(e_1,\dots ,e_4)$ 
if $E(H_r)\cap\{e_1,e_2,e_3,e_4\}$ is one of $\{e_1\}$, $\{e_4\}$, $\{e_1,e_2\}$,
$\{e_2,e_3\}$ and $\{e_3,e_4\}$.
In each of these cases it is easy to see that $E(H_r)\setminus \{e_1,\dots ,e_4\}$
can be extended to a hamiltonian cycle in $G$.

Our strategy now consists  of two parts:

\begin{itemize}
\itemsep=0pt
\item[(a)] Generate all irreducible bipartite cubic  graphs of girth 6
on up to $48$ vertices and test them for hamiltonicity.

\item[(b)] Generate all connected bipartite cubic graphs $G_r$ with at most 40 vertices
that may be the reduction of a bipartite cubic graph $G$ of girth~6.  For each 4-tuple
 $(e_1,\ldots ,e_4)$ of edges of $G_r$ that may be the new edges added in the
 reduction, determine whether $G_r$ has a hamiltonian cycle extendable for
 $(e_1,\ldots,e_4)$.  If not, reconstruct $G$ and test it for hamiltonicity.

\end{itemize}

For part (a) we used the program {\em minibaum}. 
The algorithm used in
{\em minibaum} constructs the graphs by recursively adding one edge at a time.
The subgraph constructed at each step remains part of all descendants,
so as soon as the subgraph on the left hand side of Figure~\ref{fig:8gon}
appears with conditions (i)--(iii) satisfied, we know that all cubic graphs
descended from this step are reducible.  Consequently, the generation tree
can be pruned at this point.

The program was run on a cluster with various different processors. It needed 
about 7 CPU years and generated 136,941,076 irreducible graphs.
Among the graphs were 4 non-hamiltonian ones, but they were not
3-connected. We also independently generated all irreducible graphs up to 40 vertices
using the unmodified version of {\em minibaum} to generate all bipartite
cubic graphs of girth 6 and an independently implemented program to filter the
irreducible graphs. The results were in complete agreement.

Part (b) of the proof was the most computationally expensive step in our
whole project.  We illustrate the magnitude of the task for $G_r$ having 40 vertices.
As indicated in Table~\ref{table:cubic_bip_graphs}, there are
7,439,833,931,266 connected bipartite cubic graphs with 40 vertices.
In total (up to reversal), there were 129,922,879,860,637,000 possibilities
for the 4-tuple $(e_1,e_2,e_3,e_4)$.
In all but 417,626,620,084 of these (1 in 311,098) there was an
extendable hamiltonian cycle.  Since hamiltonian cycles can be extendable
for many 4-tuples, the total number of hamiltonian cycles found was ``only''
131,062,665,710,324.
Of the 417,626,620,084 4\nobreakdash-tuples for which there was no extendable
hamiltonian cycle, the reconstructed graph on 48 vertices was hamiltonian
except in 368 cases, none of them 3-connected.
The total time was 5 CPU years for generation and 85 CPU years for
hamiltonian cycle investigation. We also independently verified the computations
for part (b) up to 34 vertices.

That completes the proof of the first part of Theorem~\ref{thm:main}.
For the second part, recall that Lemmas~\ref{lem:cyc3}--\ref{lem:girth8}
apply also to 50 vertices.
\end{proof}

\begin{remark}
For each even girth $g \geq 4$, there are infinitely many non-hamiltonian 3-connected bipartite cubic graphs of girth~$g$.
\end{remark}

\begin{proof}
There exist 3-connected bipartite cubic graphs of arbitrary large even girth $g$ (see e.g.~the survey~\cite{wormald1999}).  A non-hamiltonian 3-connected
bipartite cubic graph of arbitrary even girth $g$ can be obtained by replacing every vertex $v$ of a non-hamiltonian 3-connected bipartite cubic graph $G$ by a copy $G'_v$ of a
3-connected bipartite cubic graph $G'$ of girth $g$ with one vertex removed. If $\{v_1,v_2\}$ is an edge in $G$, then a vertex of degree $2$ in $G'_{v_1}$ is connected to a vertex of degree $2$ in $G'_{v_2}$ in a way that a 3-regular graph is constructed.

If the resulting graph has a hamiltonian cycle, it must enter and leave each $G'_v$
exactly once. Contracting each $G'_v$ to a single vertex would then recover a
hamiltonian cycle in~$G$, contradicting the assumption that $G$ is non-hamiltonian.
\end{proof}

\section{Non-hamiltonian cyclically 5-connected bipartite cubic graphs}
\label{sect:cyc5}

Cubic graphs of course cannot have connectivity greater than 3, but among
all cubic 3-connected graphs, the cyclic connectivity provides a measure of
how strong the connections between the parts of the graph are.
Though the Georges--Kelmans graph has girth~6, which is a necessary prerequisite
for being cyclically 6-connected, it is only cyclically 4-connected. To be exact:
it has no cyclic edge-cuts of size~3, but 8 cyclic edge-cuts of size~4.

Extending a folklore technique (described to us by Carol Zamfirescu)
for constructing graphs without hamiltonian paths from graphs without
hamiltonian cycles, we will now describe how to
construct non-hamiltonian cyclically 5-connected bipartite cubic graphs
out of a suitable non-hamiltonian bipartite cubic
graph with lower cyclic connectivity, such as the Georges--Kelmans graph.

Recall that minimal cyclic edge cuts in cubic graphs are always
independent edge cuts. We denote the distance between vertices $v,w$ by $d(v,w)$.
Let $G=(V,E)$ be a cyclically 4-connected cubic graph, so that there is a vertex $v \in V$ 
with neighbours $N(v)=\{w_1,w_2,w_3\}$ and an edge $e=\{x,y\}$, so that
$d(v,x)\ge 3$, $d(v,y)\ge 3$, and
for each independent edge cut $C$ of $G$ with $|C|= 4$
we have that neither $v$ nor $x$ are contained in an edge of $C$ and that $v$ and $x$ 
are in different components
of $G\setminus C$. 
Such graphs have girth at least $5$, as a 4-cycle containing $v$ or $x$ would imply
a 4-cut containing $v$ or $x$, and a 4-cycle not containing $v$ or $x$
would imply a 4-cut with $v$ and $x$ in the same component.
It is easy to see by inspection that the Georges--Kelmans graph has
in fact several vertex-edge pairs with this property, for example the pair marked
with asterisks in Figure~\ref{fig:EHGK}.
In the figures showing the operations
we use coloured vertices to illustrate that we can choose to maintain bipartiteness;
nevertheless the construction is not restricted to bipartite graphs.

\begin{figure}[h!t]
	\centering
	\includegraphics[width=0.9\textwidth]{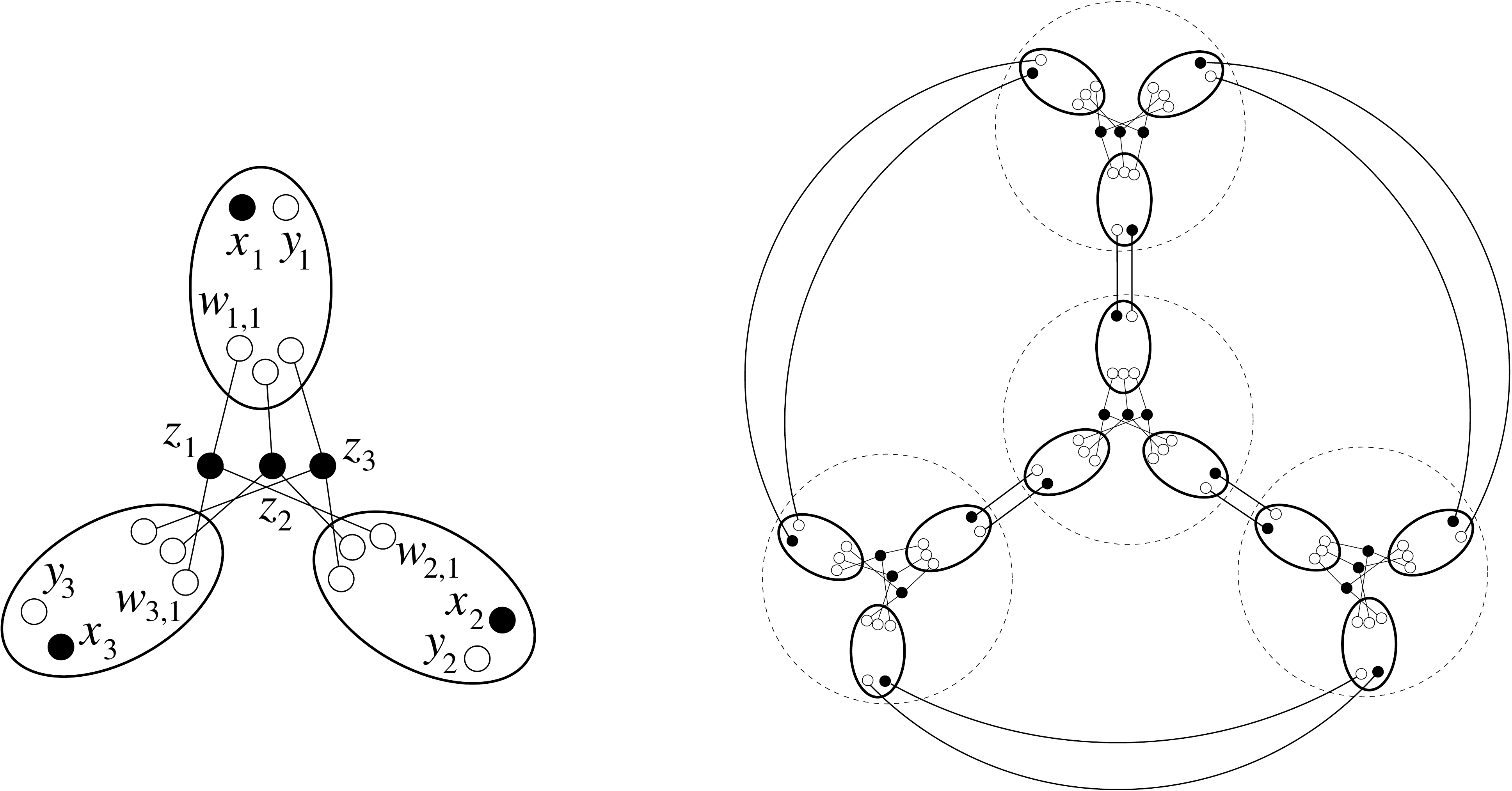}
	\caption{Constructing a cyclically 5-connected graph from a cyclically 4-connected one. The colour
of the vertices $z_i$ and the vertices $w_{i,j}$ can be the other way around. Between the vertices
$x_i$ and $y_i$ there is no edge.}
	\label{fig:cyc5}
\end{figure}

Take three copies $G_1,G_2,G_3$ of $G$ with corresponding vertices $v_i,w_{i,1},w_{i,2},w_{i,3},x_i,y_i$
for $1\le i \le 3$ and add three new vertices $z_1,z_2,z_3$. Then remove $v_1,v_2,v_3$
and for $1\le i \le 3$ and $1\le j \le 3$ connect $z_j$ with $w_{i,j}$ and remove the edge $\{x_i,y_i\}$.
This operation is depicted on the left hand side of Figure~\ref{fig:cyc5}. We call the result of
this operation a {\em triple} and each of the copies $G_i$ of $G$ with $v_i$ and $\{x_i,y_i\}$
removed a {\em brick}. Then
we can take an arbitrary 3-edge connected cubic multigraph $M$, replace each vertex by a triple
and assign the three bricks to the three adjacent edges. Finally we connect vertices
$x_k$ and $y_\ell$ from different bricks that are assigned to the same edge. On the right
hand side of Figure~\ref{fig:cyc5} this is depicted for the graph $M=K_4$, but the smallest
choice for $M$ is the cubic multigraph with $2$ vertices.
We call this operation the {\em triple-operation} $\mathcal{T}(G,v,e,M)$.

\begin{theorem}\label{thm:cyc5nonham}
Let $G=(V,E)$ be a cyclically 4-connected cubic 
non-hamiltonian graph and $v \in V$, $e=\{x,y\}\in E$ 
so that $d(v,x)\ge 3$, $d(v,y)\ge 3$, and
for each cyclic edge cut $C$ of $G$ with $|C|= 4$
we have that neither $v$ nor $x$ are contained in an edge of $C$ and that $v$ and $x$ 
are in different components of $G\setminus C$.
Furthermore let $M$ be a 3-edge-connected cubic multigraph.
Then $\mathcal{T}(G,v,e,M)$ is a cyclically 5-connected
non-hamiltonian graph.
\end{theorem}

\begin{proof}
Let $C=\{e_1,\dots ,e_4\}$ be a set of four independent edges in $\mathcal{T}(G,v,e,M)$.
We have to show that $C$ is not a cut of $\mathcal{T}(G,v,e,M)$.

\noindent\textit{Claim: } If a brick $B$ contains at most $2$ edges of $C$, then $B\setminus C$ is connected \\

\noindent\textit{Proof: }
Assume that this is not the case and let $C'\subset C$ denote a set of minimal size
so that $B\setminus C'$ is not connected.  This means that $B\setminus C'$ has two
components, say $P_1$ and $P_2$, each of which has at most two endpoints
of~$C'$.
Since every vertex of $B$ has degree as least two,
neither of $P_1$ and $P_2$ can be isolated vertices,
and the only vertices that might have degree 1 in $B\setminus C'$
are  $x,y,w_1,w_2,w_3$ (if they are also endpoints of edges in $C'$).

Since each component in $B\setminus C'$ is not an isolated vertex and has
at most two vertices of degree~1, it has a cycle unless it is a non-trivial path
within the vertices $x,y,w_1,w_2,w_3$.
As any two vertices of $x,y,w_1,w_2,w_3$
have distance at least~2 from each other ($\{x,y\}$ is removed when forming $B$,
$\{x,w_i\}$ or $\{x,w_i\}$ are not present as $d(v,x)\ge 3$, $d(v,y)\ge 3$
and an edge $\{w_i,w_j\}$ would be a triangle in $G$ contradicting that it is
cyclically 4-connected), this cannot be the case.
So both components contain a cycle.

One of $P_1,P_2$ would contain at most one of $w_1,w_2,w_3$;
w.l.o.g.\ $P_1$ contains only $w_1$ or none of these vertices. If $P_1$ contains none of
$x,y$ or both, then $C'\cup \{\{v,w_1\}\}$ (resp.\ $C'$) would be a cyclic edge-cut of
$G$ with at most $3$ edges, otherwise $C'\cup \{\{v,w_1\},\{x,y\}\}$ (resp.\ $C'\cup \{\{x,y\}\}$)
would be a cyclic edge-cut of $G$ with at most $4$ edges containing
the edge $\{x,y\}$. Both cases are in contradiction to the assumptions
on~$G$, so $B\setminus C$ is connected.

\medskip

As only one brick can contain more than two edges of $C$, for at most one brick $B$ we have that $B\setminus C$ 
is not connected. Assume first that there is such a $B\setminus C$ in a triple $T$.
Then there is at most one edge of $C$ not in $B$ and two bricks
in the same triple (different from $T$) belong to the same component of
$\mathcal{T}(G,v,e,M)\setminus C$. Due to $M$ being 3-connected
all triples different from $T$ belong to the same component. The two bricks in
$T$ different from $B$ are connected to bricks in other triples, so they
belong to the same component. 
If $B$ contains all 4 edges of $C$, then each component contains a vertex of 
$x,y,w_1,w_2,w_3$, as otherwise $C$ would be a cut in $G$ not separating $v$
and $x$. As $x,y,w_1,w_2,w_3$ are connected to the other triples, in this case
the graph is connected. If $B$ contains only 3 edges of~$C$, each component
contains at least two vertices of $x,y,w_1,w_2,w_3$ (otherwise we had a 
4\nobreakdash-cut in $G$ containing $x$ or $v$) and the result follows analogously.

In $B\setminus C$ each vertex is either in a component with
all of $w_1,w_2,w_3$---at least two of which have a path to another brick in
the triple---or with both of $x,y$ and at least one of them is connected to another 
triple. So in this case $\mathcal{T}(G,v,e,M)\setminus C$ is connected.

Assume now that all bricks are connected after removing $C$. If no triple contains three or more edges
of $C$, then all triples are connected and as $M$ is 3-edge connected (so with
the edges doubled 6-edge connected), all triples are in the same component of
$\mathcal{T}(G,v,e,M)\setminus C$.

The last case is that a triple $T$ contains three or more edges
of $C$ and is disconnected, which means that the edges are adjacent to the three
new vertices, which are nevertheless still connected to at least two bricks in the
triple. The other triples are still connected, and as each brick in $T$ is connected
with at least one edge to a brick in another triple, the bricks of $T$ also
belong to the same component.
This completes the proof of cyclic 5-connectivity. 

\smallskip

Next we prove the non-hamiltonicity.
Assume that $\mathcal{T}(G,v,e,M)$ has a hamiltonian cycle $H$. 
Each edge-cut must contain a positive even number of edges of $H$,
so the 6\nobreakdash-edge-cut isolating a triple from the rest of the graph
contains $2,4$ or $6$ edges of $H$. If it contains exactly $2$ edges,
one of the bricks has no edges going to another triple,
so it must have $2$ edges going to the vertices in the triple
not belonging to a brick---but this implies a hamiltonian cycle
in $G$.

So the edge-cut contains at least $4$ edges of $H$ and there is a brick
$B$ where both edges $e_1,e_2$ going to another triple belong to $H$.
After $H$ enters $B$ at $e_1$, it must first leave at $e_2$; otherwise
there is another segment of $H$ that enters $B$ at $e_2$ and these
two segments together would allow us to construct a hamiltonian cycle in
the non-hamiltonian graph~$G$.  The same argument applies to the
other brick which $e_1$ is incident with, giving us a cycle made of
segments of $H$ involving only two bricks, in contradiction to $H$
being a hamiltonian cycle.
\end{proof}

\begin{corollary}
There are infinitely many non-hamiltonian cyclically 5-connected  
bipartite cubic graphs. The smallest one has at most 300 vertices.
\end{corollary}

\begin{proof}
This is a direct consequence of Theorem~\ref{thm:cyc5nonham}
and the fact that the construction can be carried out to preserve bipartiteness.
The
lower bound is obtained by applying the construction to the Georges--Kelmans
graph and the cubic multigraph on $2$ vertices.
\end{proof}

Our construction raises the question of whether even higher cyclic
connectivity can be achieved in a non-hamiltonian bipartite cubic graph.
We are not able to answer this question and propose it as a research topic.



\section{Barnette's conjecture and the girth}
\label{sect:barnette_genus}

A famous variation on the problem at hand is
Barnette's conjecture~\cite{barnette1969conjecture},
which states that every 3-connected \textit{planar} bipartite cubic graph is hamiltonian.
Here our main result is the following.

\begin{theorem}\label{Barnette-bound}
Let $G$ be a 3-connected planar bipartite cubic graph with $n$ vertices.
Then
\begin{itemize}
   \item[(a)] $n\le 90$ implies that $G$ is hamiltonian;
   \item[(b)] $n\le 78$ implies that every edge of $G$ lies on a hamiltonian cycle;
   \item[(c)] $n\le 66$ implies that for any two edges $e_1,e_2$ of $G$, there
     is a hamiltonian cycle through $e_1$ but avoiding $e_2$.
\end{itemize}
\end{theorem}
\begin{proof}
  Using the generator \textit{plantri}~\cite{brinkmann07} and the program {\em cubhamg}, 
we established part (c) by direct computation. This required approximately 37 CPU years.
  We also found that the same property holds for those 3-connected planar bipartite
  cubic graphs on 68 or 70 vertices that do not have a 4-face adjacent to two other 4-faces. This required another 5 CPU years.
  Parts (a) and (b) now follow as in \cite[Theorem 5]{holton1985hamiltonian}.
\end{proof}


%
%
%
%
%

Even assuming that Barnette's conjecture is true, it is natural to ask how close to planarity
non-hamiltonian 3-connected bipartite cubic graphs can be; 
in particular, what is the minimum genus of a non-hamiltonian 
3-connected bipartite cubic graph?

Using the program {\em multi\_genus}~\cite{multigenus}
to determine the genus, we found that the Georges--Kelmans graph
has genus~5. The Ellingham--Horton graphs on 54 and 78 vertices have
genus~4 and~7, respectively.
The graphs of Horton with 92 and 96 vertices have
genus 8 or 9, and genus between 8 and~10, respectively.

During our project we compiled a collection of small non-hamiltonian
3-connected bipartite cubic graphs by a mixture of unsystematic searches.
For 50--64 vertices, the number of graphs in our collection is
1, 4, 30, 187, 1334, 3377, 29529, 204069, respectively.
(The non-hamiltonicity of these graphs was confirmed with a separate program.)
The smallest genus that occurs in the collection is~4, and 
the smallest graph found with genus~4 has 52 vertices
(see Figure~\ref{fig:52v_genus4}).

\begin{figure}[h!tb]
	\centering
	\includegraphics[width=0.45\textwidth]{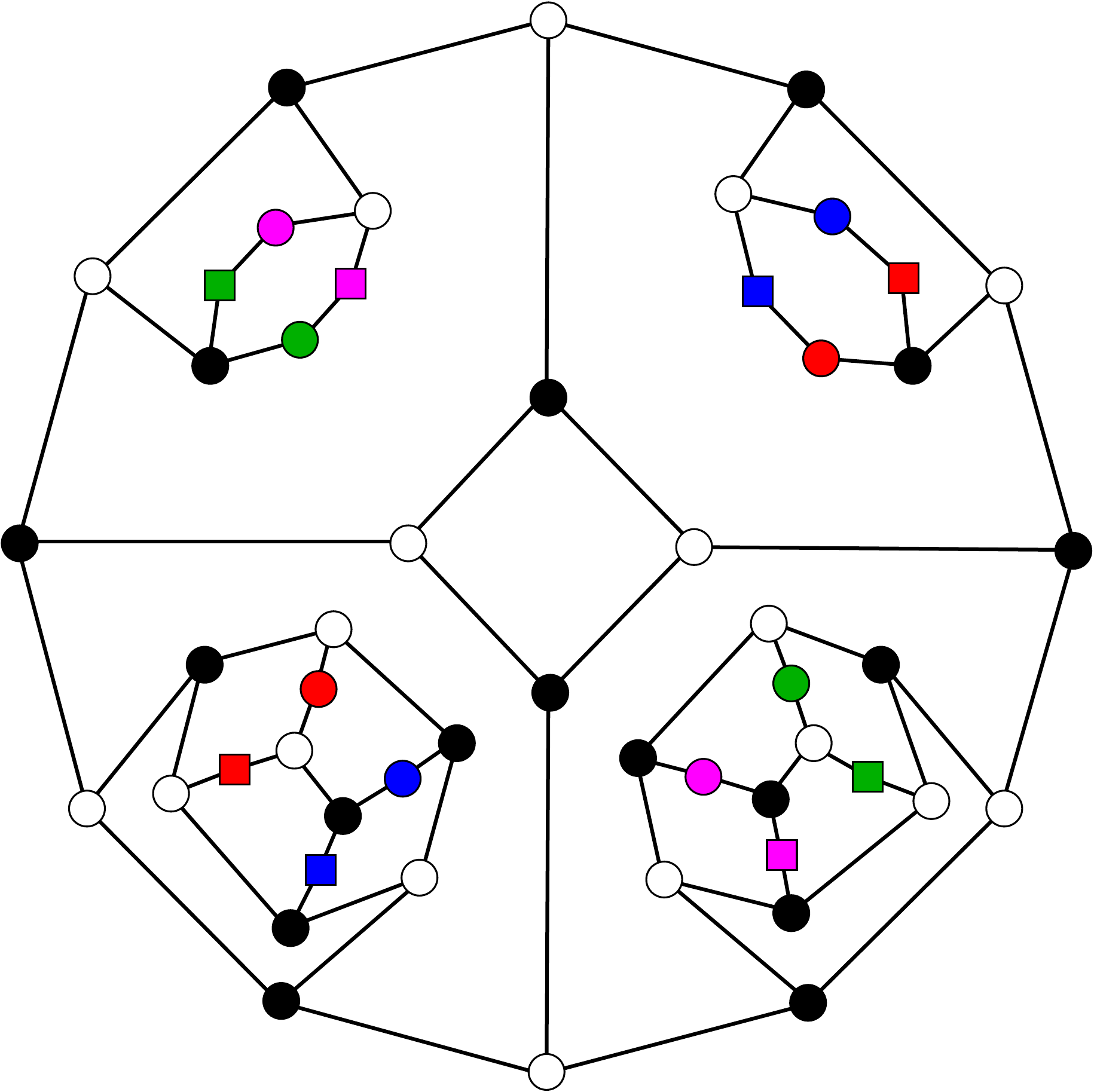}
	\caption{A non-hamiltonian 3-connected bipartite cubic graph of genus 4 on 52 vertices.
	Join two vertices of degree 2 if they have the same colour and shape. The two
	edges for each colour can be drawn on one handle.}
	\label{fig:52v_genus4}
\end{figure}




Note that the condition of being 3-connected is essential for Barnette's conjecture.
Requiring only connectivity but not 3-connectivity, there is a non-hamiltonian
planar bipartite cubic graph already on 26 vertices.
For genus 1 and 2, the smallest  bipartite cubic non-hamiltonian graphs
have 24 and 20 vertices, respectively.



\section{A fast practical algorithm for hamiltonian cycles}
\label{sect:cubhamg}

As mentioned, this project required the discovery of more than $10^{14}$ hamiltonian
cycles, in cubic graphs of up to 50 vertices.
A naive search that grows a path one edge at a time has no chance of achieving
this feat in an acceptable amount of time.
A much faster approach, available as the program {\em cubhamg} (which is part of the {\em nauty} package~\cite{nauty-website}), has been used
in several investigations since~\cite{holton1988smallest} but never published.
We now give a brief description here.

At each point of time, every edge has a label NO (not in the hamiltonian cycle),
YES (in the hamiltonian cycle), or UNDEC (undecided).  
Every edge is initialized to UNDEC, except for edges we wish to force into
or out of the hamiltonian cycle.

Given an edge labelling, a {\em propagation} process can be performed
by applying the following rules until no further rules can be applied or
a \emph{failure condition} occurs.

\begin{enumerate}
\itemsep=0pt
\item[(a)] If a vertex has two incident NOs or three incident YESes,
  a failure condition occurs.
\item[(b)] If a vertex has one incident NO, then the other two incident
  edges will be labelled YES.
\item[(c)] If a vertex has two incident YESes, then the other incident
  edge will be labelled NO.
 \item[(d)] Let $P$ be a maximal path of edges labelled YES.
   \begin{enumerate}
    \itemsep=0pt
    \item[(i)] Suppose $P$ is a hamiltonian path. If its ends are adjacent,
       then a hamiltonian cycle has been found;
       otherwise, a failure condition occurs.
    \item[(ii)] If $P$ is not a hamiltonian path and its ends are adjacent,
       then the edge between the ends will be labelled NO.
    \item[(iii)] If there are distinct vertices $x,y$ not in $P$, such that 
       the neighbours of $x$ are $y$ and the two ends of $P$, then 
       the edge $\{x,y\}$ will be labelled YES.
   \end{enumerate}
\end{enumerate}
The overall structure is a backtrack search, with three branches according
to which edges incident to an arbitrarily chosen vertex are labelled NO.
If propagation ends with a hamiltonian cycle, we are done.  If it finishes without
either a hamiltonian cycle or a failure condition,
there must be a vertex with one incident edge labelled
YES and the other two labelled UNDEC (otherwise (a), (b) or (c) could be applied).
Now the search branches into two cases depending on which of the two
UNDEC edges is to be labelled NO.
If propagation finishes with a failure condition, we backtrack to the nearest
branching point where there is an unexplored branch. 

The efficiency of this process is high because propagation usually gives
new labels to many edges.  Also, propagation can be carried out very
quickly.  Each vertex $v$ has an attribute $a_v$ whose meaning is
``if this vertex is an end of a maximal path of edges labelled YES, then
the other end is $a_v$''.  With this simple data structure, each of the
propagation operations can be carried out in constant time.

Note that the algorithm is not a heuristic for finding cycles, but
a complete search that
certifies the absence of hamiltonian cycles as well as their presence.
When a hamiltonian cycle is found, checking it is trivial, but it is also worth
pointing out that our proofs would still be valid if an error sometimes
caused a hamiltonian graph to be misidentified as non-hamiltonian.
As an example, if Table~\ref{tab:cyc4} contained a few graphs that didn't
belong, it would only mean that we have more pairs $G_1,G_2$ to join.


\section{The graphs mentioned in this paper}\label{s:hogtied}

For the reader's convenience, in Table~\ref{tab:HogIds}
we give the identification by which the graphs mentioned in
this paper can be examined and downloaded at the {\em House of Graphs}~\cite{HoG}.

%

\begin{table}[h!]
\centering
\begin{tabular}{c|c|c}
  order & description of non-hamiltonian graph& HoG id\\
  \hline
  
  20 & smallest connected bipartite cubic graph & \href{https://hog.grinvin.org/ViewGraphInfo.action?id=6923}{6923} \\
  24 & ditto, of genus 1 & \href{https://hog.grinvin.org/ViewGraphInfo.action?id=34282}{34282} \\
  26 & ditto, planar & \href{https://hog.grinvin.org/ViewGraphInfo.action?id=34286}{34286} \\
  38 & Barnette--Bosak--Lederberg graph & \href{https://hog.grinvin.org/ViewGraphInfo.action?id=954}{954} \\
  50 & Georges--Kelmans graph & \href{https://hog.grinvin.org/ViewGraphInfo.action?id=1096}{1096} \\
  52 & 3-connected bipartite cubic graph of genus 4 & \href{https://hog.grinvin.org/ViewGraphInfo.action?id=33805}{33805} \\
  54 & Ellingham--Horton graph & \href{https://hog.grinvin.org/ViewGraphInfo.action?id=1059}{1059} \\
  78 & Ellingham--Horton graph & \href{https://hog.grinvin.org/ViewGraphInfo.action?id=1061}{1061} \\
  92 & Horton graph & \href{https://hog.grinvin.org/ViewGraphInfo.action?id=1179}{1179} \\
  96 & Horton graph & \href{https://hog.grinvin.org/ViewGraphInfo.action?id=1181}{1181}  
\end{tabular}
\caption{{\em House of Graphs} id numbers for graphs mentioned in this paper.}
\label{tab:HogIds}
\end{table}
\noindent An archive of the software used for this project is available at~\cite{Programs}.


\section*{Acknowledgements}
Most of the extensive computational results were obtained from the
Stevin Supercomputer Infrastructure provided by Ghent
University, the Hercules Foundation and the Flemish Government, department EWI.



\bibliographystyle{plain}
\bibliography{references.bib}

\end{document}